\documentclass[10pt]{article}

\usepackage{graphicx}
\usepackage{cite}
\usepackage{color}
\usepackage{url}
\usepackage{amsmath}
\usepackage{amssymb}
\usepackage{amsfonts}
\usepackage{amsthm}

\usepackage{enumerate}
\usepackage{subfig}
\usepackage{fixltx2e}
\usepackage{dsfont}
\usepackage{pseudocode}
\usepackage{mathrsfs}

\usepackage{hyperref}

\newtheorem{theorem}{Theorem}[section]

\newtheorem{remark}[theorem]{Remark}
\newtheorem{definition}[theorem]{Definition}

\newtheorem{conjecture}[theorem]{Conjecture}

\newtheorem{proposition}[theorem]{Proposition}
\newtheorem{example}[theorem]{Example}
\newcommand{\EE}{\mathbb{E}}
\newcommand{\R}{\mathbb{R}}

\newcommand{\Ent}{\operatorname{Ent}}

\usepackage{eufrak}
\graphicspath{{figs/}}
\usepackage[numbers]{natbib}

\newcommand{\cP}{\mathcal{P}}
\newcommand{\cX}{\mathcal{X}}

\usepackage{todonotes}

\title{Subadditivity of the log-Sobolev constant on convolutions}
\author{Thomas~A.~Courtade~and~Edric Wang\\University of California, Berkeley}
\date{\today}
\begin{document}

\maketitle

\begin{abstract}
    We present a general subadditivity inequality for log-Sobolev constants of convolution measures.
    As a corollary, we show that the log-Sobolev constant is monotone along the sequence of standardized convolutions in the central limit theorem.
\end{abstract}

\section{Introduction}

\subsection{Background}

Let $\cP(\R^d)$ be the set of Borel probability measures on $\R^d$.
\begin{definition}
A measure $\mu\in\cP(\R^d)$ satisfies the \emph{logarithmic Sobolev} (log-Sobolev) \emph{inequality} with constant $C$  when
\begin{equation}
	\Ent_\mu(f^2) := \int f^2\log f^2\,d\mu - \int f^2\,d\mu \log \left( \int f^2\,d\mu \right) \leq 2C \int|\nabla f|^2\,d\mu
	\label{eq:lsi}
\end{equation}
for all sufficiently smooth test functions $f:\R^d\to\R$.
The smallest possible constant in the inequality is called the log-Sobolev constant and is denoted $C_{LS}(\mu)$.
If no such inequality holds, we define $C_{LS}(\mu)=\infty$.
\end{definition}

The log-Sobolev inequality is satisfied by a large class of measures.
For instance, a classical result of Bakry and \'Emery \cite{bakry2006diffusions} states that if $d\mu(x)=e^{-V(x)}\,dx$ for some $\alpha$-strongly convex potential  $V:\R^d\to\R$, then $C_{LS}(\mu)\leq1/\alpha$.
In particular, the Gaussian distributions with  largest eigenvalue of the covariance matrix equal to $\alpha$ satisfy this condition, and in this case we have $C_{LS}(\mu)=1/\alpha$ (this is Gross' original logarithmic Sobolev inequality \cite{Gross1975LOGARITHMICSI} after a change of variables).
In general, the value of the constant $C_{LS}(\mu)$ is not precisely known, but may be estimated (for example \cite[Theorem 1.1]{AIDA1994448}).
Other examples include Gaussian convolutions and mixtures of two distributions (see \cite[Section 4]{Chen2021DimensionfreeLI} for details).

Log-Sobolev constants play an important role in quantifying measure concentration phenomena.
For example, the tensorisation property \cite[Corollary 5.7]{ledoux2001concentration} implies dimension-free concentration inequalities for the product measures $\mu^n,n\geq1$ that depend only on $C_{LS}(\mu)$:
\begin{equation*}
    \Ent_{\mu^n}(f^2) \leq C_{LS}(\mu)\|f\|_\mathrm{Lip}^2
\end{equation*}
for all $f:\R^{nd}\to\R$.
Furthermore, log-Sobolev inequalities provide quantitative tail bounds for Lipschitz functions.
Herbst's argument (see, for example, \cite[Theorem 5.3]{ledoux2001concentration}) shows that the log-Sobolev inequality gives sub-Gaussian tails;
\begin{equation}
    \mu^n\left( f \geq \int f\,d\mu^n + t \right) \leq \exp\left( -\frac{t^2}{2C_{LS}(\mu)} \right)
    \label{eq:herbst}
\end{equation}
for all 1-Lipschitz functions $f:\R^{nd}\to\R$.
(All such functions are $\mu^n$-integrable so the integral on the left hand side is well-defined.)
Log-Sobolev inequalities also play an important role in Markov processes: the constant $C_{LS}(\mu)$ controls the rate of convergence of Langevin dynamics with stationary distribution $\mu$ under relative entropy \cite{bakry2014analysis}.
Analogous results hold in discrete time and are widely used in the analysis of sampling algorithms \cite{chewi2023log}.

\subsection{Main results}
\label{subsec:outline}

The main result of this paper is the following subadditivity estimate enjoyed by log-Sobolev constants:
\begin{theorem}
	\label{thm:conv}
	Let $\mu_1,\ldots,\mu_n \in \cP(\R^d)$.
	For a subset $S \subset [n]$, let $\mu_S$ be the convolution of $(\mu_i;i\in S)$.
	Then
	\begin{equation}
		C_{LS}(\mu_{[n]}) \leq \sum_{S\subset[n]} r_S \, C_{LS}(\mu_S)
        \label{eq:conv}
	\end{equation}
	for any real non-negative  coefficients  $(r_S)_{S\subset[n]}$ satisfying  $\sum_{S\ni i} r_S \geq 1$,  $\forall i\in[n]$.
\end{theorem}
Note that taking $n=2$, $r_\emptyset=r_{\{1,2\}}=0$ and $r_{\{1\}}=r_{\{2\}}=1$ recovers the following classical subadditivity result.
\begin{proposition}
    \label{prop:conv-classical}
    For ${\mu_1},{\mu_2}\in\cP(\R^d)$ satisfying a log-Sobolev inequality,
    \begin{equation*}
        C_{LS}({\mu_1}*{\mu_2}) \leq C_{LS}({\mu_1}) + C_{LS}({\mu_2}).
    \end{equation*}
\end{proposition}
As a consequence of Theorem \ref{thm:conv}, we have the following monotonicity result for the log-Sobolev constant on the sequence of convolutions in the central limit theorem:
\begin{theorem}
    \label{thm:monotone}
    Let $\mu\in\cP(\R^d)$.
    Draw $X_1,X_2,\ldots$ i.i.d.\@ from $\mu$ and denote by $\nu_n$ the distribution of the scaled sum $S_n=\frac1{\sqrt{n}}\sum_{i=1}^nX_i$.
    The sequence $(C_{LS}(\nu_n))_{n\geq1}$ is weakly decreasing in $n$.
\end{theorem}
\begin{remark}
    By Proposition \ref{prop:conv-classical} and Fekete's lemma, $\lim_{n\to\infty}C_{LS}(\nu_n)$ exists.
    Theorem \ref{thm:monotone} shows that this limit is approached monotonically.
\end{remark}

We note here that Theorem \ref{thm:conv} can be extended beyond the Euclidean setting along the same lines as \cite[Section 2.1.2]{courtade_bounds_2020}.

\subsection{Discussion and related work}

First we discuss the application of Theorem~\ref{thm:conv} to deriving log-Sobolev constants for mixture distributions.
We will need the following scaling property of the log-Sobolev constant:
\begin{proposition}
    \label{prop:scaling}
    If $Y\sim\mu$ and $\alpha Y+\beta\sim\mu_{\alpha,\beta}$ for $\alpha,\beta\in\R$, then $C_{LS}(\mu_{\alpha,\beta})=\alpha^2C_{LS}(\mu)$.
\end{proposition}
The following application is an analogue of \cite[Example 2.2]{courtade_bounds_2020}.
\begin{example}
	By convolving with a Gaussian of small variance, we can get regularity results for $\nu_n$ even when $\mu_1$ is not known to satisfy a log-Sobolev inequality.
	Let $\mu_1=\dots=\mu_n$ and let $\mu_{n+1}=\gamma_{\delta^2}$ where $\gamma_{\delta^2}$ denotes the $N(0,\delta^2I)$ distribution.
    Now let
    \begin{equation*}
        r_S=
        \begin{cases}
            1 & S=\{i,n+1\} \text{ for } i\in[n] \\
            0 &\text{ otherwise.}
        \end{cases}
    \end{equation*}
	Applying Theorem~\ref{thm:conv} and Proposition~\ref{prop:scaling} gives
    \[ C_{LS}(\nu_n\ast\gamma_{\delta^2/n}) \leq C_{LS}(\nu_1\ast\gamma_{\delta^2}) \]
    in the notation of Theorem~\ref{thm:monotone}.
	Note that the right hand side does not depend on $n$.
	So as $n$ increases, the amount of Gaussian regularisation needed to bring the log-Sobolev constant below a fixed number decreases.
\end{example}
    
To see how this improves upon existing estimates for the log-Sobolev constant of a mixture distribution \cite{Chen2021DimensionfreeLI}, let $\nu_1$ be the Rademacher distribution.
Any distribution with disconnected support (including $\nu_n$) does not have a finite log-Sobolev constant because there exists a smooth function that has zero gradient on the support of $\nu_1$ but is not constant.
However, by convolving with a Gaussian distribution we get
\[ C_{LS}(\nu_n\ast\gamma_{\delta^2/n}) = C_{LS}(\nu_1\ast\gamma_{\delta^2}) \leq 6(4+\delta^2)e^{4/\delta^2} \]
where the second inequality follows from \cite[Corollary 1]{Chen2021DimensionfreeLI}.
Applying the same result to the term $C_{LS}(\nu_n\ast\gamma_{\delta^2/n})$ directly gives
\[ C_{LS}(\nu_n\ast\gamma_{\delta^2/n}) \leq 6(4n+\delta^2/n)e^{4n^3/\delta^2} \] which has suboptimal asymptotic dependence on $n$.

Next we discuss the application of Theorems~\ref{thm:conv} and \ref{thm:monotone} to convergence in the central limit theorem.
Continue with the notation of Theorem~\ref{thm:monotone} and let $\mu\in\cP(\R^d)$ be an isotropic distribution; that is
\begin{equation*}
    \int x\,d\mu(x) = 0 \quad\text{and}\quad \int xx^T\,d\mu(x) = I_d.
\end{equation*}
The central limit theorem (CLT) states that $\nu_n$ converges in law to the standard Gaussian distribution: 
\begin{equation*}
    \gamma(dx)=(2\pi)^{-d/2}\exp(-|x|^2/2)\,dx.
\end{equation*}
Recall that $C_{LS}(\gamma)=1$ \cite{Gross1975LOGARITHMICSI}.
Also note that $C_{LS}(\mu) \geq 1$ for any isotropic distribution $\mu$; this follows from taking $f(x)=1+\epsilon x$ for small $\epsilon$ in the log-Sobolev inequality.
We can therefore view the difference $(C_{LS}(\nu_n)-1)$ as a measure of distance to Gaussianity.
(In fact this interpretation can be made quantitative, as explained below.)
This raises the question of whether $\lim_{n\to\infty}C_{LS}(\nu_n)=1$, assuming $C_{LS}(\mu)<\infty$.
For the Poincar\'e constant $C_P(\nu_n)$ (which comes from a linearisation of the log-Sobolev inequality), this result is known to be true \cite{Johnson2002ConvergenceOT}.
Thus, it is natural to propose the following.  
\begin{conjecture}
If $\mu\in\cP(\R^d)$ is isotropic and $C_{LS}(\mu)<\infty$, then $$\lim_{n\to\infty}C_{LS}(\nu_n) = 1.$$
\end{conjecture}

In addition to the conjectured asymptotic behavior, one may ask for the rate of convergence.  On this point, one might guess that the correct rate is
\begin{align}
C_{LS}(\nu_n) = 1 + O(1/n). \label{eq:conjRate}
\end{align}
Indeed, starting with the classical result of Rothaus~\cite{Rothaus1985AnalyticII}, and following with the estimates in \cite[Theorem 4.1]{Courtade2017ExistenceOS}, we have the lower bound\footnote{This bound quantitatively reinforces the interpretation of the quantity $C_{LS}(\nu_n)-1$ as a measure of the distance to Gaussianity.}
$$
C_{LS}(\nu_n)  \geq C_P(\nu_n) \geq 1 + \tfrac{1}{d} W_2^2(\nu_n,\gamma).
$$
Also by \cite[Theorem 4.1]{Courtade2017ExistenceOS}, we know that $W_2(\nu_n,\gamma)^2 \leq \tfrac{d}{n}(C_P(\mu) - 1)$, and that this rate of $O(1/n)$ is generally optimal (e.g., \cite[Theorem 1.1]{Rio2011AsymptoticCF}). Altogether, this tells us that we cannot hope to do better than \eqref{eq:conjRate}, but also suggests that rate $O(1/n)$ may be optimal.  If we assume that near-extremizers of the inequality (\ref{eq:lsi}) satisfy certain  regularity assumptions, we are able to show that $\lim_{n\to\infty} C_{LS}(\nu_n)=1$ and give a (presumably suboptimal) quantitative convergence rate, but the question remains open in general.

Finally, we remark that Theorems~\ref{thm:conv} and \ref{thm:monotone} parallel the known  monotonicity results for  Shannon entropy, Fisher information, and Poincar\'e constants \cite{Artstein2004OnTR, Artstein2004SolutionOS, courtade_bounds_2020}.  All of these  monotonicity results can be derived using  Shearer's inequality (see \cite[Remark 3.2]{courtade_bounds_2020}), which also plays a key role in our derivation of Theorem \ref{thm:conv}.  We remark that entropy and Fisher information both enjoy   associated central limit theorems \cite{Barron1986ENTROPYAT} and a rate of $O(1/n)$ is known to be optimal in these cases \cite{Johnson2001FisherII, Bobkov2011RateOC}.

\section{Proofs}
\label{sec:conv}

The proof of Theorem \ref{thm:conv} draws inspiration from the proof of \cite[Theorem 2.1]{courtade_bounds_2020}, with a key difference being that the argument remains at the level of entropy inequalities, instead of at the level of the corresponding variance inequalities obtained by linearization.
Let $\cX=\prod_{i=1}^n\cX_i$ where each $\cX_i$ is a Polish space.
For a subset $S\subset[n]$, define $\cX_S=\prod_{i\in S}\cX_i$.
Now we can define for any probability measure $P$ on $\cX$ the corresponding marginal $P_S$ on $\cX_S$ by $P_S = \pi_S\#P$, where $\pi_S:\cX\to\cX_S$ is the natural projection.
The main ingredient of the proof is Shearer's inequality:
\begin{proposition}[Shearer's inequality]
	\label{prop:shearer}
	Let $P\in\cP(\cX)$, let $Q_i\in\cP(\cX_i)$ for $i=1,\ldots,n$ and let $Q=\prod_{i=1}^nQ_i$.
	Then
	\begin{equation*}
		\sum_{S\subset [n]} c_S D(P_S\|Q_S) \leq D(P\|Q)
	\end{equation*}
	where $c_S \geq 0$ for all $S\subset [n]$ and $\sum_{S \ni i} c_S \leq 1$ for all $i\in[n]$.
\end{proposition}
We will omit the proof since it is analogous to the proof of \cite[Theorem 3.1]{courtade_bounds_2020}.
We will only need the case where $\cX_i=\R^d$ for every $i$.

For the proof of Theorem \ref{thm:conv}, it will be more convenient to work in the notation of random variables than with measures.  So, for a random variable $X$ and a non-negative measurable function $g: X\mapsto g(X)$, the entropy of $g(X)$ becomes 
$$
\Ent(g(X)) := \EE[ g(X) \log g(X) ]- \EE [g(X) ]\log \EE [g(X)]. 
$$
For a random variable $Y$  defined on the same probability space, define the conditional entropy
$$
\Ent(g(X)|Y) := \EE[ g(X) \log g(X) |Y]- \EE [g(X)|Y] \log \EE [g(X)|Y]. 
$$
Using definitions, we immediately obtain the following analogy to the classical variance decomposition:  
\begin{align}
\Ent(g(X)) = \EE[\Ent(g(X)|Y)] + \Ent(\EE[g(X)|Y])  . \label{eq:entDecomp}
\end{align}

\begin{proof}[Proof of Theorem \ref{thm:conv}]
	Let $Q=\mu_1\times\cdots\times\mu_n$ and let $Y=(Y_1,\ldots,Y_n)\sim Q$,  where $Y_i\sim\mu_i$ independently for all $i$.
	Define $U = \sum_{i=1}^nY_i$.
	For a subset $S\subset[n]$, let $Y_S=(Y_i)_{i\in S}$ and $\mu_S=\prod_{i\in S}\mu_i$.
	Also define the complement $\bar S = \{i\in[n]:i \notin S\}$.

    Fix any sufficiently smooth $f$ with $\EE[f^2(U)]=1$.  
	In the notation of Proposition \ref{prop:shearer}, we will define $P$ by
	\begin{equation*}
		\frac{dP}{dQ} = f^2(U).
	\end{equation*}
	Then
	\begin{equation*}
		\frac{dP_S}{dQ_S} = \EE[f^2(U)|Y_S]
	\end{equation*}
	and
	\begin{equation*}
		D(P_S\|Q_S) = \EE \left[ \frac{dP_S}{dQ_S}\log\frac{dP_S}{dQ_S} \right] = \Ent(\EE[f^2(U)|Y_S]),
	\end{equation*}
    where all expectations are with respect to $Q$. 
	We can therefore write Shearer's inequality as
	\begin{equation}
		\sum_{S \subset [n] } c_S \Ent(\EE[f^2(U)|Y_S]) \leq \Ent(f^2(U))
		\label{eq:shearer}
	\end{equation}
	for some non-negative $(c_S)_{S\subset [n]}$ satisfying $\sum_{S \ni i} c_S \leq 1$ for all $i$, to be specified later.
	Now we use the  entropy decomposition \eqref{eq:entDecomp} to write
	\begin{equation*}
		\Ent(f^2(U)) = \EE\left[ \Ent(f^2(U)|Y_S) \right] + \Ent\left( \EE[f^2(U)|Y_{S}] \right).
	\end{equation*}
	Summing over $S \subset [n]$ and applying equation (\ref{eq:shearer}) gives
	\begin{align*}
		\sum_{S\subset [n]} c_S \Ent(f^2(U)) &= \sum_{S\subset [n]} c_S \EE[\Ent(f^2(U)|Y_S)] + \sum_{S\subset [n]} c_S \Ent(\EE[f^2(U)|Y_{S}]) \\
		&\leq \sum_{S \subset [n]} c_S \EE[\Ent(f^2(U)|Y_S)] + \Ent(f^2(U)).
	\end{align*}
	Collecting all $\Ent(f^2(U))$ terms on the left, we may further bound 
	\begin{align}
		\left( \sum_{S\subset [n]} c_S - 1 \right) \Ent(f^2(U)) &\leq \sum_{S\subset [n]} c_S \EE[\Ent(f^2(U)|Y_S)] \nonumber \\
		&\leq 2 \sum_{S\subset [n]} c_S \EE\left[C_{LS}(\mu_{\bar S})\EE\left[|\nabla f(U)|^2\big|Y_{S}\right]\right] \nonumber \\
		&= 2 \sum_{S\subset [n]} c_S C_{LS}(\mu_{\bar S})\EE\left[|\nabla f(U)|^2\right], \label{eq:subadd-step}
	\end{align}
    where the second inequality follows by independence of $Y_S$ and $Y_{\bar{S}}$ and definition of  $C_{LS}(\mu_{\bar S})$.
    The above is homogeneous in $f^2$, so  our initial assumption $\EE[f^2(U)]=1$ may now be disposed of.

    We now choose the coefficients $(c_S)$.  Let $(r_{S})_{S\subset [n]}$ be those coefficients in the statement of the theorem satisfying   $r_{S}\geq0$ for all $S$,  and $\sum_{S\ni i} r_{S} \geq 1$ for all $i$.  If $\sum_{S \subset [n] } r_{S} = 1$, then this implies $r_{[n]}=1$ and $r_S=0$ for $S\neq [n]$, so the claimed result is trivial. Thus, we further assume that  $\sum_{S \subset [n] } r_{S} > 1$, and put $c_S = r_{\bar S}/(\sum_{S\subset[n]}r_{\bar S}-1)$.
    Then $c_S\geq0$ for all $S$,
    \begin{equation*}
        \sum_{S\ni i} c_S = \frac{\sum_{S\ni i} r_{\bar S}}{\sum_S r_{\bar S}-1} = \frac{\sum_S r_S - \sum_{S\ni i} r_{S}}{\sum_S r_{S}-1} \leq 1,
    \end{equation*}
    and
    \begin{equation*}
        \sum_{S \subset [n]} c_S = \frac{\sum_{S} r_{\bar S}}{\sum_S r_{\bar S}-1} > 1.
    \end{equation*}
    Note that $\sum_S r_{\bar S} - 1 = (\sum_S c_S - 1)^{-1}$, so we can invert the map to get $r_{\bar S} = c_S/(\sum_S c_S - 1)$.
    We may therefore substitute these coefficients $(c_S)$ into a rearranged equation (\ref{eq:subadd-step}) to conclude
    \begin{equation*}
        C_{LS}(\mu_{[n]}) \leq \sum_{S\subset[n]} \frac{c_S}{\sum_{S} c_S-1} C_{LS}(\mu_{\bar S}) = \sum_{S\subset[n]} r_{\bar S} C_{LS}(\mu_{\bar S}) = \sum_{S\subset[n]} r_S C_{LS}(\mu_S). \qedhere
    \end{equation*}
\end{proof}

With Theorem \ref{thm:conv} in hand, Theorem \ref{thm:monotone} follows.
\begin{proof}[Proof of Theorem \ref{thm:monotone}]
    Take $\mu_i=\mu$ for $i\in[n]$ and
    \begin{equation*}
        r_S=
        \begin{cases}
            \frac1{n-1} & S=\{[n]\setminus i\} \text{ for some } i\in[n] \\
            0 & \text{ otherwise.}
        \end{cases}
    \end{equation*}
    Theorem~\ref{thm:conv} then states
    \begin{equation*}
        C_{LS}(\mu_{[n]}) \leq \frac{n}{n-1} C_{LS}(\mu_{[n-1]}).
    \end{equation*}
    Applying Proposition~\ref{prop:scaling} gives the desired result.
\end{proof}

\bibliographystyle{plain}
\bibliography{functional_inequalities.bib}

\end{document}